\documentclass[11pt]{amsart}
\usepackage{amsmath,amssymb,graphicx,graphics}
\usepackage[dvipdfmx]{overpic}

\newtheorem{lem}{Lemma}[section]
\newtheorem{conj}[lem]{Conjecture}
\newtheorem{defn}[lem]{Definition}
\newtheorem{thm}[lem]{Theorem}
\newtheorem{cor}[lem]{Corollary}
\newtheorem{que}[lem]{Question}

\newtheorem{prop}[lem]{Proposition}

\def\bZ{\mathbb Z}
\def\N{\mathcal{N}}

\def\v{{\bf v}}

\begin{document}
\title{The third term in lens surgery polynomials}
\author{Motoo Tange}
\thanks{The author was partially supported by JSPS KAKENHI Grant Number 17K14180.}
\subjclass{57M25,57M27}
\email{tange@math.tsukuba.ac.jp}
\keywords{lens space surgery, Alexander polynomia, non-zero curve}
\address{Institute of Mathematics, University of Tsukuba,
 1-1-1 Tennodai, Tsukuba, Ibaraki 305-8571, Japan}
\date{}
\maketitle
\begin{abstract}
It is well-known that the second coefficient of the Alexander polynomial of any lens space knot in $S^3$ is $-1$.
We show that the non-zero third coefficient condition of the Alexander polynomial of
a lens space knot $K$ in $S^3$ confines the surgery to the one realized by the $(2,2g+1)$-torus knot,
where $g$ is the genus of $K$.
In particular, such a lens surgery polynomial coincides with $\Delta_{T(2,2g+1)}(t)$.
\end{abstract}
\section{Introduction}
\subsection{Lens space knots}
If a knot $K$ in a homology sphere $Y$ yields a lens space by an integral Dehn surgery, then 
we call $K$ a {\it lens space knot} in $Y$.
The result obtained by a Dehn surgery is written by $Y_p(K)$.
Hence, the lens space surgery is presented as $Y_p(K)=L(p,q)$.
The homology class represented by the dual knot of the surgery is identified with an element $k$ in $(\bZ/p\bZ)^\times $.
Precisely it is explained in Section~\ref{pandp}.
The pair $(p,k)$ is called a {\it lens surgery parameter}.

We call a polynomial $\Delta(t)$ {\it lens surgery polynomial (in $Y$)} if 
there exists a lens space knot $K$ in $Y$ such that $\Delta(t)=\Delta_K(t)$.
It is well-known that any lens surgery polynomials have interesting properties.

In \cite{os}, Ozsv\'ath and Szab\'o proved that any lens surgery polynomials in $S^3$ are
flat and alternating.
If the absolute values of all coefficients of a polynomial are smaller than or equal to $1$, we call the polynomial {\it flat}.
If the non-zero coefficients of a polynomial are alternating sign in order, 
then we call the polynomial {\it alternating}.
We call a polynomial $\Delta$ {\it trivial}, if $\Delta=1$.

Any lens space knot with trivial Alexander polynomial in $S^3$ is isotopic to the unknot due to \cite{KM}.
Generally, if $K$ is a lens space knot, then the degree of the Alexander polynomial coincides with the Seifert genus $g$.

In this paper we use the following notations for coefficients of any lens surgery polynomial:
$$\Delta(t)=t^{-g}\sum_{i=0}^{2g}\alpha_it^i=\sum_{i=-g}^ga_it^i.$$
In other words, this equality implies $\alpha_i=a_{i-g}$.
By the symmetry of Alexander polynomial we obtain $a_{i}=a_{-i}$ and $\alpha_{i}=\alpha_{2g-i}$.

We consider non-trivial lens surgery polynomials from now.
Then due to the author \cite{Tan1} and Hedden and Watson \cite{H}, any lens surgery polynomial in $S^3$ becomes the following form around the top coefficient $t^g$:
$$\Delta=t^g-t^{g-1}+\cdots.$$
In \cite{H}, it is shown that any L-space knot has the same form.
Namely, the second coefficient from the top has $-1$.

In \cite{Tan1}, it is proven that the second top coefficient of any lens space knot in any L-space homology sphere is $-1$.
\subsection{Third top term of lens space polynomial}
Let $T(p,q)$ be the right-handed $(p,q)$-torus knot.
Teragaito asked an interesting question about the next coefficient:
\begin{que}
\label{tera}
If a non-trivial lens surgery polynomial in $S^3$ has the following form:
$$\Delta=t^g-t^{g-1}+t^{g-2}+\cdots,$$
then does $\Delta$ coincide with $\Delta_{T(2,2g+1)}$?

In other words, if a lens surgery polynomial is not a $(2,2g+1)$-torus knot polynomial for some integer $g$,
then $\alpha_2=0$?.
\end{que}
Here we give an affirmative answer for this question.
\begin{thm}
\label{main}
Teragaito's question (Question~\ref{tera}) is true.
\end{thm}
This theorem is true even if the lens space knot $K$ lies in an L-space homology sphere and satisfies $2g(K)\le p$ by applying the same method.

Theorem 1.15 in \cite{Tan1} gave a criterion for a lens space knot $K$ to satisfy $\Delta_K(t)=\Delta_{T(2,2g+1)}(t)$ for some positive integer $g$.
On the other hand, we can also say that Theorem~\ref{main} gives a new criterion for a lens space knot to have the same Alexander polynomial as that of $T(2,2g+1)$.
\subsection{Realization of lens surgery}
We define the following terminology. 
\begin{defn}
Let $p,k$ be relatively prime positive integers.
If a lens surgery $Y_p(K)=L(p,q)$ in a homology sphere $Y$ has the lens surgery parameter $(p,k)$, then we say that the parameter $(p,k)$ is realized by a lens space knot $K$.

\end{defn}

\begin{cor}
\label{cor}
Let $K$ be a lens space knot in $S^3$ with the surgery parameter $(p,k)$.
The Alexander polynomial $\Delta_K(t)$ has the following form:
$$\Delta_K=t^g-t^{g-1}+t^{g-2}+\cdots,$$
if and only if $(p,k)$ is realized by $T(2,2g+1)$.
\end{cor}
This condition in this corollary is equivalent to the condition of $k=2$.
%
\subsection{The cases of lens space knots $K_{p,k}$ in $Y_{p,k}$}
\label{KpkYpk}
Consider a simple $(1,1)$-knot in a lens space yielding a homology sphere by some integer slope.
The `simple' is defined in \cite{st} and \cite{Tan3}.
If such a $(1,1)$-simple knot generates the 1st homology of the lens space, 
we can always find such a slope.
Hence any simple $(1,1)$-knot is parameterized by a relatively prime integers $(p,k)$. 
The dual knot is a lens space knot in the homology sphere.
The dual knot is denoted by $K_{p,k}$ and the homology sphere by $Y_{p,k}$.
The reader should probably understand these facts by reading \cite{st} and \cite{Tan3}.
The main result in \cite{IST} gave a formula of the Alexander polynomial of $K_{p,k}$ by using $p,k$.
Here we give the following conjecture:
\begin{conj}
\label{11knot}
If the third top term of the symmetrized Alexander polynomial $\Delta_{K_{p,k}}(t)$ is non-zero, then $\Delta_{K_{p,k}}(t)$ coincides with $\Delta_{T(2,2g+1)}(t)$ for some integer $g$, in other words, $k=2$ holds.
\end{conj}
This conjecture can be easily checked by a computer program ($p\le 600$) based on the formula in \cite{IST}.
Conjecture~\ref{11knot} is true under a little strong condition that $Y_{p,k}$ is homeomorphic to $S^3$, because of Theorem~\ref{main}.
The essential point is what if the third top term of $\Delta_{K_{p,k}}$
is non-zero, then $Y_{p,k}$ is homeomorphic to $S^3$.
Notice that in \cite{Tan1} the author proved that $k=2$ holds if and only if $Y_{p,k}$ is homeomorphic to $S^3$ and $K_{p,k}$ is isotopic to $T(2,2g+1)$ for some integer $g$.
This condition is also equivalent to the equality $\Delta_{K_{p,k}}(t)=\Delta_{T(2,2g+1)}(t)$.
\section*{Acknowledgements}
Question \ref{tera} is presented by Masakazu Teragaito's talk in the Mini-symposium ``Knot Theory on Okinawa'' at OIST on February the 17th to 21st, 2020.
I appreciate him for teaching us this question.

\section{Preliminaries and Proofs}
\label{pandp}
\subsection{Brief preliminaries}
Here we define the lens surgery parameter $(p,k)$.
\begin{defn}
Let $K$ be a knot in a homology sphere $Y$.
Suppose that $Y_p(K)=L(p,q)$ and the dual knot $\tilde{K}$ has $[\tilde{K}]=k[c]\in H_1(L(p,q),{\mathbb Z})$ for some orientation of $\tilde{K}$.
Here the dual knot is the core knot in the solid torus obtained by the Dehn surgery.
Furthermore, $c$ is either of core circles of genus one Heegaard decomposition of $L(p,q)$.
Then we call $(p,k)$ {\it lens surgery parameter}.
The integer $k$ is called a {\it dual class}.
\end{defn}
If $L(p,q)$ is a Dehn surgery of a homology sphere, the surgery parameter $(p,k)$ is relatively prime and $q=k^2\bmod p$.
Note that we adopt the orientation of $L(p,q)$ as the $p/q$-surgery of the unknot in $S^3$.

The ambiguity of the orientation of $\tilde{K}$ and the choices of the core circles of genus one Heegaard decomposition give (at most) four possibilities of the dual class $k_0,-k_0,k_0^{-1},-k_0^{-1}$ (in ${\mathbb Z}/p{\mathbb Z}$), for some integer $k_0$.
We always take the minimal integer $k$ as a representative satisfying $0< k<p/2$.

For any integer $i$ we define the integer $[i]_p$ to be the integer with $i\equiv [i]_p\bmod p$ and $-\frac{p}{2}<[i]_p\le \frac{p}{2}$.
Let $k_2$ be the absolute value of the integer $[k']_p$ satisfying $kk'\equiv 1\bmod p$.
We call $k_2$ {\it the second dual class}.
We set $kk_2\equiv e\bmod p$, $e=\pm 1$, $m=\frac{kk_2-e}{p}$, $q=[k^2]_p$, $q_2=[(k_2)^2]_p$.
$c=\frac{(k-1)(k+1-p)}{2}$ and for some non-zero integer $\ell$
$$I_\ell:=\begin{cases}\{1,2,\cdots, \ell\}&\ell>0\\\{\ell+1,\cdots, -1,0\}&\ell<0.\end{cases}$$
From these data, we can compute the coefficient $a_i$ due to \cite{Tan2}.
\begin{prop}[Proposition 2.3 in \cite{Tan1}]
Let $K$ be a lens space knot in $S^3$.
For any integer $i$ with $|i|\le p/2$, the $i$-th coefficient of the Alexander polynomial
$$a_i=-e(m-\#\{j\in I_k|[q_2(j+ki+c)]_p\in I_{ek_2}\}).$$
\end{prop}

To prove Theorem~\ref{main}, we use the {\it non-zero curve} defined in \cite{Tan1}.
First, we extend the coefficients $a_i$ of the Alexander polynomial periodically as $\bar{a}_i=a_{[i]_p}$.
By the estimate in \cite{Gr} proven in \cite{KM}, $\bar{a}_i$ is determined,
where $g(K)$ is the Seifert genus of $K$.

We define {\it $A$-matrix} and {\it $dA$-matrix} as follows:
$$A_{i,j}=\bar{a}_{k_2(i+jek-c)},\ dA_{i,j}=A_{i,j}-A_{i-1,j},$$
where $c=(k-1)(k+1-p)/2$.
Due to the formula (9) in Lemma 2.6 in \cite{Tan1}, we have
\begin{equation}
\label{daformua}
dA_{i,j}=E_{ek_2}(q_2i+k_2(j+e))-E_{ek_2}(q_2i+k_2j)=\begin{cases}1&[q_2i+k_2j]_p\in I_{-k_2}\\-1&[q_2i+k_2j]_p\in I_{k_2}\\0&\text{otherwise.}\end{cases}
\end{equation}

We put $A_{i,j}$ on each lattice point $(i,j)$ in ${\mathbb Z}^2\subset {\mathbb R}^2$.
For a non-zero coefficient $A_{i,j}$ we draw a horizontal positive or negative arrow on $(i,j)$ according to $A_{i,j}=1$ or $-1$ respectively, where a positive (or negative) arrow means a horizontal arrow with positive (or negative) in the $i$-direction.
After that, we connect the horizontally adjacent arrows with the same orientation
and compatibly connect arrows around the non-zero $dA_{i,j}$ as in \cite{Tan1}.
Then we can obtain an infinite family of simple curves on ${\mathbb R}^2$ with no finite ends (i.e., they are properly embedded curves in ${\mathbb R}^2$). 
The arrows are not-increasing with respect to the $j$-coordinate.
We call the curves {\it non-zero curves}. 
\begin{prop}[\cite{Tan1}]
\label{mainp}
Any non-zero curve for any lens space knot in $S^3$ is included in a non-zero region $\mathcal{N}$.
In each non-zero region there is a single component non-zero curve.
\end{prop}
Here a non-zero region $\mathcal{N}$ (introduced in \cite{Tan1}) is defined as follows.
First, we consider the union of $2g+1$ box-shaped neighborhoods of a vertical sequent lattice points corresponding to $\alpha_0,\alpha_1\cdots, \alpha_{2g}$.
Two adjacent box neighborhoods are overlapped with a horizontal unit segment.
Next, we take the infinite parallel copies moved by $n\cdot\v$ where $\v$ is the vector $(1,-k_2)$ and $n$ is any integer.
We denote the union of the infinite parallel copies of $\N$ and call it {\it non-zero region}.
Moving a non-zero region $\N$ by $n\cdot (0,p)$ for any integer $n$, we obtain 
infinite non-zero regions on ${\mathbb R}^2$.

The following lemma is important to prove the main theorem.
This is also the case of $m=0$ in Lemma 4.4 in \cite{Tan1}.
\begin{lem}
\label{le}
If there exist integers $i_0,j_0$ such that $dA_{i_0,j_0}=-dA_{i_0,j_0+1}=-1$, then for any integer $i$, there are no two adjacent zeros in the sequence $\{dA_{i,s}|s\in {\mathbb Z}\}$.
\end{lem}
\begin{proof}
We assume the existence of integers $i_0,j_0$.
Let $x$ be $i_0+(j_0-1)ek$.
Using the formula (\ref{daformua}),
we have $[q_2x]_p\in I_{-k_2}$, $[q_2x+k_2]_p\in I_{k_2}$, and $[q_2x+2k_2]_p\in I_{-k_2}$.
Hence, the sequence $[q_2x+sk_2]_p$ starts at $[q_2x]_p$ and returns in $I_{-k_2}$ at $s=2$.
Therefore, we have $p-k_2<(q_2x+2k_2)-q_2x<p+k_2$ and this means $p< 3k_2$.

We suppose $dA_{i,j}=dA_{i,j+1}=0$ for some integers $i,j$.
Then $[q_2(i+jek)]_p,[q_2(i+jek)+k_2]_p\not\in I_{-k_2}\cup I_{k_2}$.
This implies $p-k_2-k_2\ge k_2$.
This contradicts the inequality above.

If for an integer $I$, the sequence $\{dA_{I,s}|s\in {\mathbb Z}\}$ has no adjacent zeros, 
for any integer $i$ the same thing holds because the $\{dA_{i,s}|s\in {\mathbb Z}\}$ is a parallel copies of $\{dA_{I,s}|s\in {\mathbb Z}\}$.
Hence, the desired condition is satisfied.
\hfill$\Box$
\end{proof}
Note that this lemma holds for any relatively prime positive integers $(p,k)$.
Actually, to prove this lemma we do not require that the matrices $A$ and $dA$ come from a lens space knot in $S^3$.
In particular, if any $K_{p,k}$ in $Y_{p,k}$ (defined in Section~\ref{KpkYpk}) has non-zero third term in the Alexander polynomial, then 
$p<3k_2$ holds.
To prove Conjecture~\ref{11knot}, first we should probably classify $(Y_{p,k},K_{p,k})$ in the case of $3k_2<p$.
\subsection{Proof of Theorem \ref{main}}
Let $K$ be a lens space knot with lens surgery parameter $(p,k)$ and with $g=g(K)$.
Suppose that $\alpha_0=1$, $\alpha_{1}=-1$, and $\alpha_{2}=1$.
Now we assume that $2g-k_2\ge 3$.
Since any non-zero curve has no finite ends, $\alpha_{4}=-1$ holds naturally.
Hence, we can assume that
$$\begin{cases}\alpha_0=1\\\alpha_{1}=-1\\\alpha_{3}=1\\\alpha_{4}=-1.\end{cases}\hspace{1cm}(\ast)$$

Let $i,j$ be fixed integers with $k_2(i+jek-c)=-g\bmod p$.
Then $A_{i,j}=\alpha_0=1$.
$A_{i-1,j}=A_{i-1,j+1}=A_{i-1,j+2}=A_{i-1,j+3}=0$, because any non-zero curve is included in a non-zero region $\N$ due to Proposition~\ref{mainp}.

We notice that the assumption of Lemma~\ref{le} is satisfied.
Thus, we have $dA_{i,j}=1$, $dA_{i,j+1}=-1$, $dA_{i,j+2}=1$, and $dA_{i,j+3}=-1$.
The local values for matrices $A$ and $dA$ are drawn in the top pictures in Figure \ref{cases}.

Our situation falls into the following two cases (I), and (II) as in Figure~\ref{on}.\\
{\bf (I)}: $A_{i+1,j+1}=-1$, and $A_{i+1,j+2}=1$\\
{\bf (II)}: $A_{i+1,j+1}=0$, and $A_{i+1,j+2}=0$.\\

In the case of (I), we obtain $dA_{i+1,j+1}=dA_{i+1,j+2}=0$.
This contradicts Lemma \ref{le}.

Next, consider the case of (II).
We claim $A_{i+2,j+1}=A_{i+2,j+2}=0$.
If $A_{i+2,j+1}$ or $A_{i+2,j+2}$ is non-zero, then the non-zero term is included in the non-zero region right next to $\N$, because there is only one non-zero curves in any non-zero region (Proposition~\ref{mainp}).
This implies that by seeing the vertical coordinate in ${\mathbb R}^2$, we have
$$p-2k_2\le 2.$$
Since $2k_2<p$, we have $p=2k_2+1$ or $2k_2+2$.
The equality $p=2k_2+1$ means it gives a $(2,2g+1)$-torus knot surgery.
We consider the case of $p=2k_2+2$.
Since $p,k_2$ are relatively prime, $k_2$ is an odd number.
The equality $p=2k_2+2$ can be deformed into $k_2^2-1=\frac{k_2-1}{2}p\equiv 0\bmod p$.
It is a lens space surgery yielding $L(p,1)$.
This is the $k_2=1$ case only due to \cite{KM}.
Thus the claim above is true.

The remaining cases are $-2\le -2g+k_2\le 1$.
By using Theorem 1.15 and Theorem 4.20 in \cite{Tan1},
the cases are realized by the $(2,2g+1)$-torus knot surgeries or
the lens surgery on $T(3,4)$ or $Pr(-2,3,7)$. 
The knots $T(3,4)$ and $Pr(-2,3,7)$ both do not satisfy $(\ast)$.
Thus, our remaining cases satisfying this condition $(\ast)$
give the equality $\Delta_K(t)=\Delta_{T(2,2g+1)}$.
\hfill$\Box$
\begin{figure}[htb]
\begin{overpic}
[width=0.5\textwidth]
{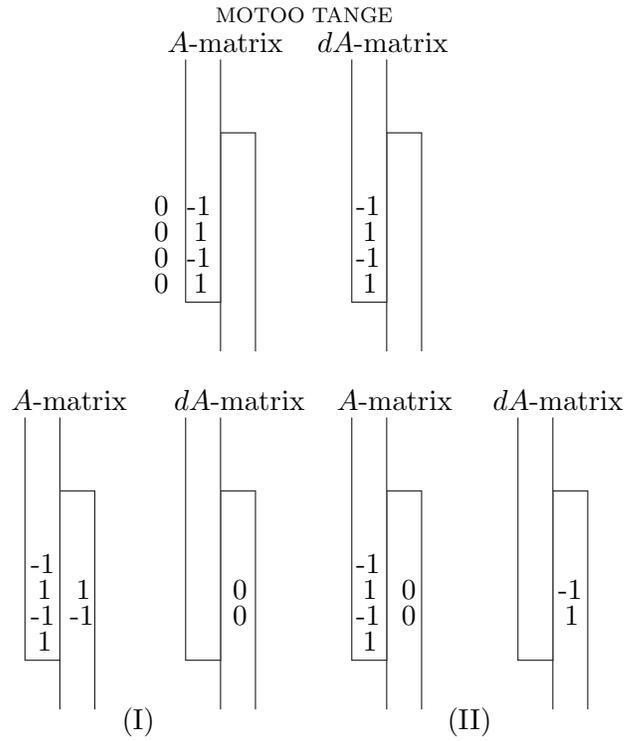}
\put(22,101){$A$-matrix}
\put(26,64){1}
\put(25,68){-1}
\put(26,72){1}
\put(25,76){-1}
\put(20,64){0}
\put(20,68){0}
\put(20,72){0}
\put(20,76){0}

\put(45,101){$dA$-matrix}

\put(51,76){-1}
\put(52,72){1}
\put(51,68){-1}
\put(52,64){1}

\put(-2,46){$A$-matrix}
\put(1,21){-1}
\put(2,17){1}
\put(1,13){-1}
\put(2,9){1}
\put(8,17){1}
\put(7,13){-1}

\put(15,-3){(I)}

\put(23,46){$dA$-matrix}		

\put(32,17){0}
\put(32,13){0}

\put(48,46){$A$-matrix}
\put(51,21){-1}
\put(52,17){1}
\put(51,13){-1}
\put(52,9){1}
\put(58,17){0}
\put(58,13){0}

\put(65,-3){(II)}

\put(72,46){$dA$-matrix}

\put(82,17){-1}
\put(83,13){1}


\end{overpic}
\caption{Case (I) and (II)}
\label{cases}
\end{figure}

\begin{figure}[htb]
\begin{overpic}
[width=0.1\textwidth]
{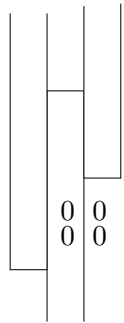}

\put(16,32){0}
\put(16,24){0}
\put(26,32){0}
\put(26,24){0}
\end{overpic}
\caption{An $A$-matrix of (II)}
\label{on}
\end{figure}

\subsection{Proof of Corollary~\ref{cor}}
We give a proof of Corollary~\ref{cor}.
Let $K$ be a lens space knot in $S^3$ with parameter $(p,k)$.
If $\Delta_K(t)$ has $\alpha_0=-\alpha_1=\alpha_2=1$, then 
$\Delta_K(t)=\Delta_{T(2,2g+1)}$ holds by Theorem \ref{main}.
Using Theorem 1.15, the lens surgery parameter is $(p,2)$.
The parameter is realized by a $(2,2g+1)$-torus knot.
\hfill$\Box$

\end{document}